\documentclass[12pt,reqno]{amsart}
\usepackage{amscd,amsmath,amsthm,amssymb}
\usepackage{color}
\usepackage{stmaryrd}
\usepackage{tikz}
\usepackage{url}

\usepackage{latexsym}
\usepackage{amsfonts,amsmath,mathtools}
\usepackage{graphics}
\usepackage{float}
\usepackage{enumitem}

\definecolor{verylight}{gray}{0.97}
\definecolor{light}{gray}{0.9}
\definecolor{medium}{gray}{0.85}
\definecolor{dark}{gray}{0.6}

\numberwithin{equation}{section}

\def\NZQ{\mathbb}               

\def\ZZ{{\NZQ Z}}

\def\G{{\mathcal G}}

\def\pd{\textup{proj}\phantom{.}\!\textup{dim}}

\def\opn#1#2{\def#1{\operatorname{#2}}} 
\opn\chara{char} \opn\length{\ell} \opn\pd{pd} \opn\rk{rk}
\opn\projdim{proj\,dim} \opn\injdim{inj\,dim} \opn\rank{rank}
\opn\depth{depth} \opn\grade{grade} \opn\height{height}
\opn\embdim{emb\,dim} \opn\codim{codim}

\opn\Tr{Tr} \opn\bigrank{big\,rank}
\opn\superheight{superheight}\opn\lcm{lcm}
\opn\trdeg{tr\,deg}
\opn\reg{reg} \opn\lreg{lreg} \opn\ini{in} \opn\lpd{lpd}
\opn\size{size} \opn\sdepth{sdepth}
\opn\link{link}\opn\fdepth{fdepth}\opn\lex{lex}
\opn\tr{tr}
\opn\type{type}
\opn\indeg{indeg}
\opn\gap{gap}
\opn\diam{diam}
\opn\Mod{Mod}
\opn\div{div} \opn\Div{Div} \opn\cl{cl} \opn\Cl{Cl}
\opn\Spec{Spec} \opn\Supp{Supp} \opn\supp{supp} \opn\Sing{Sing}
\opn\Ass{Ass} \opn\Min{Min}\opn\Mon{Mon}
\opn\Ann{Ann} \opn\Rad{Rad} \opn\Soc{Soc}
\opn\Im{Im} \opn\Ker{Ker} \opn\Coker{Coker} \opn\Am{Am}
\opn\Hom{Hom} \opn\Tor{Tor} \opn\Ext{Ext} \opn\End{End}
\opn\Aut{Aut} \opn\id{id}

\opn\nat{nat}
\opn\pff{pf}
\opn\Pf{Pf} \opn\GL{GL} \opn\SL{SL} \opn\mod{mod} \opn\ord{ord}
\opn\Gin{Gin} \opn\Hilb{Hilb}\opn\sort{sort}
\opn\PF{PF}\opn\Ap{Ap}
\opn\dist{dist}
\opn\aff{aff}
\opn\relint{relint} \opn\st{st}
\opn\lk{lk} \opn\cn{cn} \opn\core{core} \opn\vol{vol}  \opn\inp{inp} \opn\nilpot{nilpot}
\opn\link{link} \opn\star{star}\opn\lex{lex}\opn\set{set}
\opn\width{wd}
\opn\Fr{F}
\opn\QF{QF}
\opn\G{G}
\opn\type{type}\opn\res{res}
\opn\conv{conv}
\opn\sr{sr}
\opn\gr{gr}

\def\pot#1#2{#1[\kern-0.28ex[#2]\kern-0.28ex]}

\opn\dirlim{\underrightarrow{\lim}}
\opn\inivlim{\underleftarrow{\lim}}
\let\union=\cup

\let\Union=\bigcup

\def\Implies{\ifmmode\Longrightarrow \else
	\unskip${}\Longrightarrow{}$\ignorespaces\fi}
\def\implies{\ifmmode\Rightarrow \else
	\unskip${}\Rightarrow{}$\ignorespaces\fi}
\def\iff{\ifmmode\Longleftrightarrow \else
	\unskip${}\Longleftrightarrow{}$\ignorespaces\fi}

\let\:=\colon
\newtheorem{Theorem}{Theorem}[section]
\newtheorem{Lemma}[Theorem]{Lemma}

\newtheorem{Proposition}[Theorem]{Proposition}
\newtheorem{Remark}[Theorem]{Remark}

\newtheorem{Example}[Theorem]{Example}

\newtheorem{Definition}[Theorem]{Definition}

\newtheorem*{Question}{Question}

\let\epsilon\varepsilon
\let\kappa=\varkappa
\def\qed{\ifhmode\textqed\fi
	\ifmmode\ifinner\hfill\quad\qedsymbol\else\dispqed\fi\fi}
\def\textqed{\unskip\nobreak\penalty50
	\hskip2em\hbox{}\nobreak\hfill\qedsymbol
	\parfillskip=0pt \finalhyphendemerits=0}
\def\dispqed{\rlap{\qquad\qedsymbol}}

\opn\dis{dis}
\def\pnt{{\raise0.5mm\hbox{\large\bf.}}}

\opn\Lex{Lex}
\opn\Shad{Shad}

\begin{document}
\title{Macaulay's theorem for vector-spread algebras}
\author{Marilena Crupi, Antonino Ficarra, Ernesto Lax}

\address{Marilena Crupi, Department of mathematics and computer sciences, physics and earth sciences, University of Messina, Viale Ferdinando Stagno d'Alcontres 31, 98166 Messina, Italy}
\email{mcrupi@unime.it}

\address{Antonino Ficarra, Department of mathematics and computer sciences, physics and earth sciences, University of Messina, Viale Ferdinando Stagno d'Alcontres 31, 98166 Messina, Italy}
\email{antficarra@unime.it}

\address{Ernesto Lax, Department of mathematics and computer sciences, physics and earth sciences, University of Messina, Viale Ferdinando Stagno d'Alcontres 31, 98166 Messina, Italy}
\email{erlax@unime.it}

\subjclass[2020]{Primary 05E40; Secondary 13A02,13D02, 13D40, 13F55.}
	
\keywords{vector-spread monomial ideals, vector-spread strongly stable ideals, Macaulay's Theorem, Kruskal--Katona's Theorem, Hilbert functions, Betti numbers.}

\begin{abstract}
	Let $S=K[x_1,\dots,x_n]$ be the standard graded polynomial ring, with $K$ a field, and let ${\bf t}=(t_1,\dots,t_{d-1})\in\ZZ_{\ge 0}^{d-1}$, $d\ge 2$, be a $(d-1)$-tuple whose entries are non negative integers. To a {\bf t}-spread ideal $I$ in $S$, we associate a unique $f_{\bf t}$-vector and we prove that if $I$ is {\bf t}-spread strongly stable, then there exists a unique {\bf t}-spread lex ideal which shares the same $f_{\bf t}$-vector of $I$ \emph{via} the combinatorics of the {\bf t}-spread shadows of special sets of monomials of $S$. Moreover, we characterize the possible $f_{\bf t}$-vectors of {\bf t}-vector spread strongly stable ideals generalizing the well-known theorems of Macaulay and Kruskal--Katona. Finally, we prove that among all {\bf t}-spread strongly stable ideals with the same $f_{\bf t}$-vector, the {\bf t}-spread lex ideals have the largest Betti numbers. 
\end{abstract}

\maketitle

\section{Introduction}
One of the main well--studied and important numerical invariant of a graded ideal in a standard graded polynomial ring is its Hilbert function which gives the sizes of the graded components of the ideal. There is an extensive literature on this topic (see, for instance, \cite{JT} and the references therein). Usually, Hilbert functions are described using the well--known Macaulay's expansion with binomials. This fact often implies the use of combinatorial tools and furthermore the arguments consist of very clever computations with binomials. The crucial idea of Macaulay is that there exist special monomial ideals, the so called \emph{lex ideals}, that attain all possible Hilbert functions. The pivotal property is that a lex ideal grows as slowly as possible. The ``squarefree'' analogue of Macaulay’s theorem is known as the Kruskal--Katona theorem. Indeed, if Macaulay’s theorem describes the possible Hilbert functions of the graded ideals in polynomial rings, the possible $f$-vectors of a simplicial complex are characterized in the theorem of Kruskal--Katona \cite{JT, GK, JK}. In fact, the Hilbert function of the Stanley--Reisner ring of a simplicial complex $\Delta$ is determined by the $f$-vector of $\Delta$, and vice versa. Kruskal--Katona's theorem is a fundamental result in topological combinatorics and discrete geometry which quickly have aroused much interest in face enumeration questions for various classes of simplicial complexes, polytopes, and manifolds. Furthermore, it may be also interpreted as a theorem on Hilbert functions of quotients of exterior algebras \cite{AHH}. 
The lex ideals as well as squarefree lex ideals play a key role in the study of the minimal free resolutions of monomial ideals. Indeed, if one considers the stable and squarefree stable ideals and the formulas for computing their graded Betti numbers \cite{JT} one can deduce the Bigatti-Hulett theorem \cite{BAM, HH} which says that lex ideals have the largest graded Betti numbers among all graded ideals with the same Hilbert function (see also \cite{HH2, KP}). 

Let $S=K[x_1,\dots,x_n]$ be the standard graded polynomial ring, with $K$ a field, and let ${\bf t}=(t_1,\dots,t_{d-1})\in\ZZ_{\ge0}^{d-1}$, $d\ge 2$,  be a $(d-1)$-tuple whose entries are non negative integers. Recently in \cite{F1}, the class of ${\bf t}$\textit{-spread strongly stable ideals} has been introduced. It is a special class of monomial ideals which generalizes the class of $t$-spread strongly stable ideals in \cite{EHQ} (see also, \cite{ACF, DHQ} and the references therein), $t$ non negative integer. More in detail, a monomial $u=x_{j_1}x_{j_2}\cdots x_{j_\ell}$ ($1\le j_1\le j_2\le\cdots\le j_\ell\le n$) of degree $\ell\le d$ of $S$ is called a \textit{vector-spread monomial of type ${\bf t}$} or simply a \textit{${\bf t}$-spread monomial} if $j_{i+1}-j_{i}\ge t_{i}$, for $i=1,\dots,\ell-1$ and a \textit{${\bf t}$-spread monomial ideal} is a monomial ideal of $S$ generated by ${\bf t}$-spread monomials. A \textit{${\bf t}$-spread strongly stable ideal} is a \textit{${\bf t}$-spread monomial ideal} with an additional combinatorial property (Definition \ref{def:stronglylex}). For ${\bf t} = (0, \ldots, 0), (1, \ldots, 1)$, one obtains the classical notions of strongly stable ideal and squarefree strongly stable ideal, respectively \cite{JT}. 
The aim of this article is to generalize Macaulay's theorem for the class of ${\bf t}$-spread strongly stable ideals. The crucial role is played by the class of \textit{${\bf t}$-spread lex ideal} (Definition \ref{def:stronglylex}). Since, ${\bf t} \in \ZZ_{\ge0}^{d-1}$, \emph{i.e.}, the entries of ${\bf t}$ can be also zero, in order to unify the theory about the classification of the Hilbert functions of graded ideals of $S$, we put our attention on the classification of the possible $f_{{\bf t}}$-vector (Definition \ref{def:ft-vector}) of ${\bf t}$-spread strongly stable ideals. More in detail, we answer to the following question:

\begin{Question}
	Under which conditions  a given sequence of positive integers
	\[
	f = (f_{-1}, f_{0}, \ldots, f_{d-1}),
	\]
	is the $f_{{\bf t}}$-vector of a ${\bf t}$-spread strongly stable ideal?
\end{Question}
 
The paper is organized as the following. Section \ref{sec1} contains some preliminaries and notions that will be used in the article. We introduce the  notion of \textit{${\bf t}$-spread shadow} of a set of monomials of $S$ and the notions of ${\bf t}$\textit{-spread strongly stable set (ideal)}  and ${\bf t}$\textit{-spread lex set (ideal)}  (Definitions \ref{def:stronglylexset}, \ref{def:stronglylex}). The combinatorics of such sets is deeply analyzed in Section \ref{sec2}. The key result in the section is Theorem \ref{Thm:BayerVectSpread} which allows us to prove that to every ${\bf t}$-spread strongly stable ideal $I$ one can associate a unique ${\bf t}$-spread lex ideal which has the same $f_{\bf t}$-vector  as $I$ (Corollary \ref{Cor:SubstituteTLex}). Moreover, we point out why this is not possible for an arbitrary ${\bf t}$-spread monomial ideal. Section \ref{sec3} contains the main result in the article which gives the classification of all suitable $f_{\bf t}$-vectors of a ${\bf t}$-spread strongly stable ideal (Theorem \ref{thm:main}). The classification is obtained by introducing a new operator (Definition \ref{Def:vector-spreadOp}) which, for suitable values of ${\bf t}$, is analog either to the operator $a\longrightarrow a^{\langle d\rangle}$ or to the operator $a\longrightarrow a^{(d)}$ which are involved in Macaulay's theorem and in Kruskal--Katona's theorem, respectively \cite{JT}. Finally, in Section \ref{sec4}, as an application of the results in the previous sections, we state an upper bound for the graded Betti numbers of the class of all ${\bf t}$-spread strongly stable ideals with a given $f_{{\bf t}}$-vector (Theorem \ref{thm:upperbound}). We prove that the ${\bf t}$-spread lex ideals give the maximal Betti numbers among all ${\bf t}$-spread strongly stable ideals with a given $f_{{\bf t}}$-vector. Such a statement generalizes the well--known result proved independently by Bigatti \cite{BAM} and Hulett \cite{HH} for graded ideals in a polynomial ring over a field of characteristic zero and afterwards generalized by Pardue \cite{KP} to any characteristic. The article contains some examples illustrating the main results developed using  \emph{Macaulay2} \cite{GDS}.

\section{Preliminaries and notations}\label{sec1}
Let $S=K[x_1,\dots,x_n]$ be the standard graded polynomial ring, with $K$ a field, and let ${\bf t}=(t_1,\dots,t_{d-1})\in\ZZ_{\ge0}^{d-1}$, $d\ge 2$, be a $(d-1)$-tuple whose entries are non negative integers. A monomial $u=x_{j_1}x_{j_2}\cdots x_{j_\ell}$ ($1\le j_1\le j_2\le\cdots\le j_\ell\le n$) of degree $\ell\le d$ of $S$ is called a \textit{vector-spread monomial of type ${\bf t}$} or simply a \textit{${\bf t}$-spread monomial} if $j_{i+1}-j_{i}\ge t_{i}$, for $i=1,\dots,\ell-1$.

If $I$ is a graded ideal of $S$ we denote by $I_j$ the $j$-graded component of $I$ and by $\indeg(I)$ the initial degree of $I$, \emph{i.e.}, the smallest $j$ such that $I_j\neq 0$. Moreover, for a monomial ideal $I\subset S$, we denote by $G(I)$ the unique minimal set of monomial generators. Furthermore, if $j\ge0$ is an integer, we set $G(I)_j=\{u\in G(I):\deg(u)=j\}$.

A \textit{${\bf t}$-spread monomial ideal} is a monomial ideal of $S$ generated by ${\bf t}$-spread monomials.
For instance, $I=(x_1x_4^2x_5,x_1x_4^2x_6,x_1x_5x_7)$ is a $(3,0,1)$-spread monomial ideal of the polynomial ring  $S=K[x_1, \ldots, x_7]$, but it is not $(3,0,2)$-spread as $x_1x_4^2x_5\in G(I)$ is not a $(3,0,2)$-spread monomial.

Note that any monomial (ideal) is ${\bf 0}$-spread, where ${\bf 0}=(0,0,\dots,0)$. If $t_i\ge1$, for all $i$, a ${\bf t}$-spread monomial (ideal) is a \textit{squarefree} monomial (ideal). 

We denote by $M_{n,\ell,{\bf t}}$ the set of all ${\bf t}$-spread monomials of degree $\ell$ in $S$. If $\ell \leq d$, by \cite[Corollary 2.4]{F1},
\begin{equation}\label{Formula:|Mn,l,t|}
\lvert M_{n,\ell,{\bf t}} \rvert = \binom{n+(\ell-1)-\sum_{j=1}^{\ell-1}t_j}{\ell}.
\end{equation}
	
\begin{Definition}\em \label{def:vector-spread.shadow}
	For a set $L$ of monomials of $S$, one defines the \textit{vector-spread shadow} or simply the \textit{${\bf t}$-spread shadow} of $L$ to be the set
	\begin{align*}
	\Shad_{\bf t}(L)\ &=\ \big\{wx_j:w\in L, wx_j\ \textup{is}\ {\bf t}\textup{-spread}, \, j=1,\dots,n \big\}.
	\end{align*}
\end{Definition}
	
Note that $\Shad_{\bf t}(M_{n,\ell,{\bf t}})=M_{n,\ell+1,{\bf t}}$ for all $\ell\ge 0$ and that $\Shad_{\bf t}(L)=\emptyset$ whenever all monomials in $L$ have degrees $\ge d$. Moreover, one can quickly observe that if $L$ is a set of monomials of $S$, then the definition of $\Shad_{\bf 0}(L)$ coincides with the classical notion of shadow of $L$ \cite[Chapter 6]{JT}.

\begin{Definition}\em\label{def:ft-vector}
	If $I$ is a monomial ideal of $S$, we denote by $[I_j]_{\bf t}$ the set of all ${\bf t}$-spread monomials in $I_j$. Furthermore, we set 
	\[
	f_{{\bf t},\ell-1}(I) = \vert M_{n,\ell,{\bf t}}\vert - \vert [I_\ell]_{\bf t}\vert, \qquad 0\le\ell\le d.
	\]
	and define the vector 
	\[
	f_{\bf t} (I)= (f_{{\bf t},-1}(I), f_{{\bf t},0}(I), \ldots, f_{{\bf t}, d-1}(I)).
	\]
	Such a vector is called the \textit{$f_{\bf t}$-vector} of $I$. Note that $f_{{\bf t},-1}(I)=1$.
\end{Definition}

\begin{Remark}\em
	For ${\bf t} =(t_1, \ldots, t_{d-1})$ with $t_i\ge 1$, for all $i$, then $I$ is the Stanley--Reisner ideal $I_\Delta$ of a uniquely determined simplicial complex $\Delta$ on vertex set $\{1,\ldots,n\}$ with $f_{\bf 1}(I)$ as $f$-vector, where ${\bf 1} =(1, \ldots, 1)$.
\end{Remark}

\begin{Definition}\em \label{def:stronglylexset}
	Let $L\subseteq M_{n,\ell,{\bf t}}$. $L$ is called a ${\bf t}$\textit{-spread strongly stable set} if for all $u\in L$, $j<i$ such that $x_i$ divides $u$ and $x_j(u/x_{i})$ is ${\bf t}$-spread, then $x_j(u/x_{i})\in L$. $L$ is called a ${\bf t}$-\textit{spread lex set}, if for all $u\in L$, $v\in M_{n,\ell,{\bf t}}$ such that $v\ge_{\lex}u$, then $v\in L$. 
\end{Definition}
	
Here $\ge_{\lex}$ stands for the lex order induced by $x_1>\dots>x_n$ \cite{JT}.
	
For our convenience, throughout the article, we assume the empty set to be both a ${\bf t}$-spread strongly stable set and a lex set.
	
\begin{Definition}\em \label{def:stronglylex}
	Let $I$ be a ${\bf t}$\textit{-spread ideal}. $I$ is said to be a ${\bf t}$\textit{-spread strongly stable ideal} if $[I_{\ell}]_{\bf t}$ is a ${\bf t}$\textit{-spread strongly stable set}, for all $\ell$.

	$I$ is said to be a ${\bf t}$\textit{-spread lex ideal} if $[I_{\ell}]_{\bf t}$ is a ${\bf t}$\textit{-spread lex set}, for all $\ell$.
\end{Definition}
	
One can observe that any {\bf t}-spread lex set (ideal) is a {\bf t}-spread strongly stable set (ideal). Moreover, for ${\bf t}={\bf 0}$ (${\bf t}={\bf 1}=(1,\ldots,1)$) one obtains the classical notions of (squarefree) strongly stable ideal and (squarefree) lex ideal \cite{JT}.\\

\section{Combinatorics on vector-spread shadows}\label{sec2}
In this section, if  ${\bf t} =(t_1, \ldots, t_{d-1})\in\ZZ_{\ge0}^{d-1}$, $d\ge 2$, we deal with the combinatorics of the {\bf t}-spread shadows of {\bf t}-spread strongly stable sets and {\bf t}-spread lex sets. As a consequence, given a  {\bf t}-spread strongly stable ideal $I$ of the polynomial ring $S$, we prove the existence of a unique {\bf t}-spread lex ideal of $S$ with the same $f_{\bf t}$-vector of $I$.\\

For  $i,j$ integers, we set $[i,j]=\{k\in\ZZ: i\leq k\leq j\}$. Note that $[i,j]\ne\emptyset$ if and only if $i\le j$.
 
Let $u$ be a monomial. We set $\max(u)=\max\{i:x_i\ \textup{divides}\ u\}$.

\begin{Lemma}\label{Lemma:ShadVectSS}
	Let $L\subseteq M_{n,\ell,{\bf t}}$ be a ${\bf t}$-spread strongly stable set. Then
	\[
	\Shad_{\bf t}(L)=\big\{wx_j:w\in L,\ j\ge\max(w),\ wx_j\ \textup{is}\ {\bf t}\textup{-spread}\big\}.
	\]
\end{Lemma}
\begin{proof}
	Let $u\in\Shad_{\bf t}(L)$. Then, $u=w x_j $ for some $w\in L$. If $\max (w)\leq j$ there is nothing to prove. Suppose $j<\max(w)$. We can write $u=x_{\max (u)}w'$, with $w' = x_j(u/x_{\max (u)})$ and $\max(w')\le\max(u)$. The proof is complete if we show that $w'\in L$. Let $u=x_{j_1}\cdots x_{j_{\ell+1}}$ with $j_1\le\dots\le j_{\ell+1}$. Then, $j=j_p$ for some $p<\ell+1$ and $w'=x_{j_1}\cdots x_{j_\ell}=x_{j_p}(u/x_{\max(u)})$ is a {\bf t}-spread monomial because $w$ is {\bf t}-spread. Moreover, $w'\in L$ since $j_p<\max(u)=j_{\ell+1}$ and $L$ is a {\bf t}-spread strongly stable set.
\end{proof}
	
\begin{Proposition}\label{Prop:Shadt(L)}
	Let $L\subseteq M_{n,\ell,{\bf t}}$ be a ${\bf t}$-spread set. 
	\begin{enumerate}
		\item[\textup{(a)}] If $L$ is a ${\bf t}$-spread strongly stable set, then $\Shad_{\bf t}(L)\subseteq M_{n,\ell+1,{\bf t}}$ is a ${\bf t}$-spread strongly stable set.
		\item[\textup{(b)}] If $L$ is a ${\bf t}$-spread lex set, then $\Shad_{\bf t}(L)\subseteq M_{n,\ell+1,{\bf t}}$ is a ${\bf t}$-spread lex set.
	\end{enumerate}
\end{Proposition}
\begin{proof}
	Let $u=w x_j\in\Shad_{\bf t}(L)$. For the proofs of both (a) and (b), by Lemma \ref{Lemma:ShadVectSS}, we can assume that $\max(w)\le j$.\medskip\\
	(a) Let $i<k$ such that $x_k$ divides $u$ and $u'=(u/x_k)x_i$ is {\bf t}-spread. We prove that $u'\in\Shad_{\bf t}(L)$. If $k=j$, then $u'=w x_i\in\Shad_{\bf t}(L)$ by definition. Suppose $k\ne j$. Since $j=\max(u)$, then $i<k<j$ and consequently $x_k$ divides $w$. Therefore,\linebreak
	$u'=x_i(w/x_k)x_j\in\Shad_{\bf t}(L)$ because $x_i(w/x_k)\in L$ as $x_i(w/x_k)$ is a {\bf t}-spread monomial, $i<k$ and $L$ is a {\bf t}-spread strongly stable set.\smallskip\\
	(b) Let $v\in M_{n,\ell+1,{\bf t}}$ with $v>_{\lex}u$. We prove that $v\in\Shad_{\bf t}(L)$. By definition of the lex order, it follows that $v/x_{\max(v)}\ge_{\lex} u/x_{\max(u)}$. The hypothesis on $L$ guarantees that $v/x_{\max(v)}\in L$. Hence, $v=(v/x_{\max(v)})x_{\max(v)}\in\Shad_{\bf t}(L)$.
\end{proof}

Let $L\subseteq M_{n,\ell,{\bf t}}$ be a set of monomials, where $\ell\leq d$. For every $i\in[1,n]$ we denote by $m_i (L)$ the number of monomials $u\in L$ such that $\max(u)=i$ and then we set $m_{\leq j}(L) = \sum_{i=1}^{j} m_i(L)$. Note that $m_i(L)=0$ if $i\leq \sum_{j=1}^{\ell-1} t_j$.

\begin{Lemma}\label{Lemma:m_i}
	Let $L\subseteq M_{n,\ell,{\bf t}}$ be a $\bf t$-spread strongly stable set with $\ell<d$. Then
	\begin{enumerate}
		\item[\textup{(a)}] $m_i(\Shad_{\bf t}(L))=m_{\leq i-t_\ell}(L)$ for all $i$;
		\item[\textup{(b)}] $\left\lvert \Shad_{\bf t}(L) \right\rvert = \sum_{k=1+\sum_{j=1}^{\ell-1} t_j}^{n-t_\ell} m_{\leq k}(L)$.
	\end{enumerate}	
\end{Lemma}	
\begin{proof}
	(a) If $i\leq\sum_{j=1}^{\ell} t_j$ the proof is trivial. Let $i\geq 1+\sum_{j=1}^{\ell} t_j$. Consider the map
	\[\varphi : \left\{ u\in\Shad_{\bf t}(L)\, :\, \max(u)=i \right\} \longrightarrow \left\{ w \in L \, :\, \max(w)\leq i-t_\ell \right\},\]
	defined as follows. Let $u\in\Shad_{\bf t}(L)$ with $\max(u)=i$. By Lemma \ref{Lemma:ShadVectSS}, $u=w x_i$ where $w\in L$ is the unique monomial such that $\max(u)=i$. Thus, we set $\varphi(u)=w$. The map $\varphi$ is well defined by the uniqueness of $w$. To prove (a), it is enough to show that $\varphi$ is a bijection. $\varphi$ is clearly injective. To prove that $\varphi$ is surjective, let $w\in L$ with $\max(w)\leq i-t_\ell$. Then, $u=w x_i$ is ${\bf t}$-spread because $\max(w)\leq i - t_\ell$. Since $\max(u)=i$, then $u$ belongs to the domain of $\varphi$ and $\varphi(u)=w$, as desired.\medskip
	
	\noindent
	(b) Since
	$$
	\Shad_{\bf t}(L)=\bigcup_{i=1+\sum_{j=1}^{\ell}t_j}^{n}\big\{u\in\Shad_{\bf t}(L):\max(u)=i\big\},
	$$
	where the union is disjoint, by (a), $|\{u\in\Shad_{\bf t}(L):\max(u)=i\}|=m_i(\Shad_{\bf t}(L))=m_{\le i-t_\ell}(L)$, and so
	\begin{align*} 
	\left\lvert \Shad_{\bf t}(L) \right\rvert & = \sum_{i=1+\sum_{j=1}^{\ell}t_j}^{n} m_i (\Shad_{\bf t}(L)) \\
		& = \sum_{i=1+\sum_{j=1}^{\ell}t_j}^{n} m_{\leq i-t_\ell} (L) \\
		& = \sum_{k=1+\sum_{j=1}^{\ell-1}t_j}^{n-t_\ell} m_{\leq k} (L).
	\end{align*}
\end{proof}
	
The following result is a vector-spread analogue of a well known theorem due to Bayer, see \cite[Theorem 6.3.3]{JT}.
		
\begin{Theorem}\label{Thm:BayerVectSpread}
	Let $L\subset M_{n,\ell,{\bf t}}$ be a {\bf t}-spread lex set and let $N\subset M_{n,\ell,{\bf t}}$ be a {\bf t}-spread strongly stable set. Suppose $\left\lvert L \right\rvert \leq \left\lvert N \right\rvert$. Then, $m_{\leq i}(L)\leq m_{\leq i}(N)$.
\end{Theorem}
\begin{proof}
	We first observe that $N=N_0 \union N_1 x_n$, where $N_0$ and $N_1$ are the unique {\bf t}-spread strongly stable sets such that
	\begin{align*}
		N_0\ &=\ \{u\in N:\max(u)<n\},&
		N_1\ &=\ \{u/x_n:u\in N,\max(u)=n\}.
	\end{align*}
	
	Similarly, we can write $L=L_0\cup L_1x_n$, where $L_0$ and $L_1$ are ${\bf t}$-spread lex sets defined as above.\smallskip
	
	We proceed by induction on $n\ge1$, with the base case being trivial. Let $n>1$. Firstly, observe that $m_{\leq n}(L)=\left\lvert L \right\rvert$ and $\left\lvert N \right\rvert = m_{\leq n}(N)$. Hence, the assertion holds for $i=n$. Note that $m_{\leq n-1}(L)=|L_0|$ and $m_{\leq n-1}(N)=|N_0|$. Thus, to say that $m_{\leq n-1}(L)\leq m_{\leq n-1}(N)$ is equivalent to prove that
	\begin{equation}\label{eq:N0L0Crucial}
		|L_0|\leq|N_0|.
	\end{equation}
	Assume for a moment that inequality \eqref{eq:N0L0Crucial} holds. Then, applying our inductive hypothesis to the sets $L_0,N_0\subset M_{n-1,\ell,{\bf t}}$ we obtain
	$$
	m_{\le i}(L)=m_{\le i}(L_0)\le m_{\le i}(N_0)=m_{\le i}(N)\ \ \ \text{for}\ \ \ i=1,\dots,n-1,
	$$
	as desired. Thus, it remains to prove the inequality \eqref{eq:N0L0Crucial}. \smallskip
	
	Let $N_0^* \subset M_{n-1,\ell,{\bf t}}$ be a {\bf t}-spread lex set with $\left\lvert N_0^* \right\rvert = \left\lvert N_0 \right\rvert$ and $N_1^* \subset M_{n-t_{\ell-1},\ell-1,{\bf t}}$ be a {\bf t}-spread lex set with $\left\lvert N_1^* \right\rvert = \left\lvert N_1 \right\rvert$. Let $N^* = N_0^*\cup N_1^* x_n$. We claim that $N^*$ is a {\bf t}-spread strongly stable set. Let $u\in N^*$. We shall prove that for every $j<i$ such that $x_i$ divides $u$ and $x_j(u/x_{i})$ is ${\bf t}$-spread, then $x_j(u/x_{i}) \in N^*$. If $u\in N_0^*$ there is nothing to prove since $N_0^*$ is a ${\bf t}$-spread lex set. Suppose $u\in N_1^* x_n$, then we can write $u=wx_n$, where $w\in N_1^*$. If $i<n$, then $w'= x_j(w/x_{i})$ belongs to $N_1^*$ and $x_j(u/x_i)=w'x_n\in N_1^* x_n$. If $i=n$, then $x_j(u/x_i)=wx_j$. Now, if $x_n$ divides $wx_j$, then again $x_jw\in N_1^*x_n$. Otherwise, if $x_n$ does not divide $wx_j$, then $wx_j\in N^*$ if and only if $wx_j\in N_0^*$. Thus, we must show that $\Shad_{\bf t}(N_1^*)\subset N_0^*$. For this aim, it is sufficient to prove that $|\Shad_{\bf t}(N_1^*)|\leq |N_0^*|$, as both sets are ${\bf t}$-spread lex sets (Proposition \ref{Prop:Shadt(L)}(b)). By Lemma \ref{Lemma:m_i} and the induction hypothesis we obtain
	\begin{align*}
	|\Shad_{\bf t}(N_1^*)| & = \sum_{i=1+\sum_{j=1}^{\ell-2}t_j}^{n-t_{\ell-1}} m_{\leq i}(N_1^*)\leq \sum_{i=1+\sum_{j=1}^{\ell-2}t_j}^{n-t_{\ell-1}} m_{\leq i}(N_1) \\
	& = |\Shad_{\bf t}(N_1)|\leq |N_0|= |N_0^*|.
	\end{align*}
	Finally, $N^*$ is a {\bf t}-spread strongly stable set.\smallskip
	
	Since $|N|=|N^*|$, we may replace $N$ by $N^*$ and assume that $N_0$ is a {\bf t}-spread lex set. We suppose $n>1 + \sum_{j=1}^{\ell-1} t_j$, otherwise $M_{n,\ell,{\bf t}}=\{x_1x_{1+t_1}\cdots x_{1+\sum_{j=1}^{\ell-1} t_j}\}$ and the assertion is trivial.\smallskip
	
	Let $m=x_{j_1}\cdots x_{j_{\ell}}$ be a {\bf t}-spread monomial and $\alpha : M_{n,\ell,{\bf t}} \rightarrow M_{n,\ell,{\bf t}}$ be the map defined as follows:
	\begin{enumerate}
		\item[{\normalfont(a)}] if $j_{\ell} \neq n$, then $\alpha(m)=m$;
		\item[{\normalfont(b)}] if $j_{\ell} = n$ and $m\neq \min_{>_{\lex}}M_{n,\ell,{\bf t}}=x_{n-\sum_{j=1}^{\ell -1}t_j} x_{n-\sum_{j=2}^{\ell -1}t_j}\cdots x_{n-t_{\ell -1}}x_n$, then there exists $r\in [2,\ell]$ such that $j_r > j_{r-1} + t_{r-1}$. Hence, if $r$ is the largest integer with this property, we define
		$$\alpha(m) = x_{j_1}\cdots x_{j_{r-1}} x_{j_r - 1}\cdots x_{j_{\ell -1}-1}x_{n-1};$$
		\item[{\normalfont(c)}] if $j_{\ell} = n$ and $m=\min_{>_{\lex}}M_{n,\ell,{\bf t}}=x_{n-\sum_{j=1}^{\ell -1}t_j} x_{n-\sum_{j=2}^{\ell -1}t_j}\cdots x_{n-t_{\ell -1}}x_n$, then
			$$\alpha(m)=x_{n-1-\sum_{j=1}^{\ell -1}t_j} x_{n-1-\sum_{j=2}^{\ell -1}t_j}\cdots x_{n-1-t_{\ell -1}}x_{n-1}.$$
	\end{enumerate}
	Such map $\alpha$ is well defined and is easily seen to be a lexicographic order preserving map, \textit{i.e.}, if  $m_1,m_2\in M_{n,\ell,{\bf t}}$ and $m_1 <_{\lex} m_2$, then $\alpha(m_1)<_{\lex}\alpha(m_2)$, too.\smallskip
	
	To prove \eqref{eq:N0L0Crucial}, since both $L_0$ and $N_0$ are {\bf t}-spread lex sets, it is enough to show that $\min_{>_{\lex}} L_0 \geq_{\lex} \min_{>_{\lex}} N_0$.
	
	Let $u=\min_{>_{\lex}} L=x_{i_1}\cdots x_{i_{\ell}}$ and $v=\min_{>_{\lex}} N=x_{j_1}\cdots x_{j_{\ell}}$. We claim that $\alpha(u)=\min_{>_{\lex}}L_0$ and $\alpha(v)=\min_{>_{\lex}}N_0$. Indeed,
	\begin{enumerate}
		\item[{\normalfont(a)}] if $v\in N_0$, then $\alpha(v)=v\in N_0$
		\item[{\normalfont(b)}] if $v\in N_1 x_n$ and $\alpha(v)=x_{j_1}\cdots x_{j_{r-1}} x_{j_r - 1}\cdots x_{j_{\ell -1}-1}x_{n-1}$, where $r\in [2,\ell]$ is the largest integer such that $j_r > j_{r-1} + t_{r-1}$, then
		\begin{enumerate}
			\item[{\normalfont(i)}] if $r=\ell$, then $\alpha(v)=x_{j_1}\cdots x_{j_{\ell-1}} x_{n-1}=(v/x_n)x_{n-1}\in N_0$, because $N$ is a {\bf t}-spread strongly stable set.
			\item[{\normalfont(ii)}] if $r<\ell$, since $N$ is a {\bf t}-spread strongly stable set, we have
			\begin{align*}
			v_1&= x_{j_k -1}(v/x_{j_k})\in N,\\
			v_2&= x_{j_{k+1} -1}(v_1/x_{j_{k+1}})\in N,\\
			&\phantom{..}\vdots\\
			v_{\ell-r}&= x_{j_{\ell-1} -1}(v_{\ell -k-1}/x_{j_{\ell-1}})\in N,
			\end{align*}
			then $\alpha(v)=x_{n-1}(v_{\ell -k}/x_n)\in N_0$.
		\end{enumerate}
		\item[{\normalfont(c)}] if $v\in N_1 x_n$ and $\alpha(v)=x_{j_1 -1}\cdots x_{j_{\ell -1}-1}x_{n-1}$, we have
		\begin{align*}
		v_1&=x_{j_1 -1}(v/x_{j_1})\in N,\\
		v_2&=x_{j_2 -1}(v_1/x_{j_2})\in N,\\
		&\phantom{..}\vdots\\
		v_{\ell-1}&=x_{j_{\ell} -1}(v_{\ell-2}/x_{j_{\ell}-1})\in N,
		\end{align*}
		and $\alpha(v)=x_{n-1}(v_{\ell -1}/x_n)\in N_0$, because $N$ is a {\bf t}-spread strongly stable set.
	\end{enumerate}
	
	Hence, in all possible cases $\alpha(v)\in N_0$. Thus, we have $\min_{>_{\lex}}N_0 \leq_{\lex} \alpha(v)$, and $\min_{>_{\lex}}N_0 \geq_{\lex} v =\min_{>_{\lex}}N$. Since $\max(\min_{>_{\lex}}N_0)<n$, we have 
	$$
	\min_{>_{\lex}}N_0 = \alpha(\min_{>_{\lex}} N_0) \geq_{\lex} \alpha(v) \geq_{\lex} \min_{>_{\lex}} N_0,
	$$ 
	and so $\min_{>_{\lex}} N_0 = \alpha(v)$. Similarly one can prove that $\min_{>_{\lex}} L_0 = \alpha(u)$.
	
	Finally, since $L$ is a {\bf t}-spread lex set and $|L|\leq |N|$, we have $u\geq_{\lex}v$. Consequently, $\min_{>_{\lex}}L_0 = \alpha(u) \geq_{\lex} \alpha(v)= \min_{>_{\lex}}N_0$ and the proof is completed.
\end{proof}

As a striking consequence of the previous result, we prove that to every ${\bf t}$-spread strongly stable ideal $I$ one can associate a unique ${\bf t}$-spread lex ideal which shares the same $f_{\bf t}$-vector of $I$. It is necessary to highlight that if $I$ is an arbitrary ${\bf t}$-spread ideal of $S$, then such a ${\bf t}$-spread lex ideal does not always exist as pointed out in the next example.

\begin{Example}\em{(\cite[Remark 2]{CAC})}
	Let $I = (x_2 x_8, x_2 x_6, x_2 x_4) \subset K[x_1,\ldots,x_8]$, which is a ${\bf t}$-spread ideal, with ${\bf t} = (2)$. Such $I$ is not strongly stable, and we have the impossibility to construct a ${\bf t}$-spread lex ideal with the same $f_{\bf t}$-vector as $I$.
\end{Example}
	
Nevertheless, there can exist a ${\bf t}$-spread ideal $I$ of $S$ which is not ${\bf t}$-spread strongly stable but for which there exists a ${\bf t}$-spread lex ideal with the same $f_{\bf t}$-vector of $I$ (see, for instance, \cite[Remark 4.10]{ACF} and \cite[Remark 2]{CAC}).

\begin{Proposition}\label{Cor:SubstituteTLex}
	Let $I\subset S$ be a ${\bf t}$-spread strongly stable ideal. Then, there exists a unique ${\bf t}$-spread lex ideal $I^{{\bf t},\lex}\subset S$ such that $f_{\bf t}(I)=f_{\bf t}(I^{{\bf t},\lex})$.
\end{Proposition}
\begin{proof}
	We construct a ${\bf t}$-spread lex ideal $J$ verifying $f_{\bf t}(J)=f_{\bf t}(I)$ as follows. For all $\ell\in[0,d]$, let $L_\ell$ be the unique {\bf t}-spread lex set of $M_{n,\ell,{\bf t}}$ with $|L_\ell|=|[I_\ell]_{\bf t}|$. Whereas, for $\ell>d$ we set $L_{\ell}=\emptyset$. For all $\ell\ge0$, we denote by $J_\ell$ the $K$-vector space spanned by the monomials in the set
	$$
	L_\ell\cup\Shad_{\bf 0}(B_{\ell-1}),
	$$
	where $B_{-1}=\emptyset$ and for $\ell\ge1$, $B_{\ell-1}$ is the set of monomials in $J_{\ell-1}$. 
	
	Then, we set $J=\bigoplus_{\ell\ge0}J_{\ell}$. We claim that $J$ satisfies our statement. Firstly, we must show that $J$ is a ${\bf t}$-spread lex ideal. For this purpose, it is enough to observe that $\Shad_{\bf 0}(B_{\ell-1})\subseteq B_\ell$ for all $\ell\ge1$ by construction
	
	It remains to prove that $f_{\bf t}(J)=f_{\bf t}(I)$, \emph{i.e.}, $|[J_\ell]_{\bf t}|=|[I_{\ell}]_{\bf t}|$ for all $\ell\in[0,d]$. Let $\delta=\indeg(I)= \indeg(J)$. Then, $\delta\le d$ and $|[J_\ell]_{\bf t}|=|[I_{\ell}]_{\bf t}|=0$ for all $\ell\in[0,\delta-1]$. Now, let $\ell\ge\delta$. Since $|L_{\ell}|=|[I_\ell]_{\bf t}|$, then
	$$
	|[J_\ell]_{\bf t}|=|L_\ell\cup\Shad_{\bf t}(B_{\ell-1})|=|[I_\ell]_{\bf t}|
	$$
	if and only if $\Shad_{\bf t}(B_{\ell-1})\subseteq L_\ell$. We proceed by finite induction on $\ell\in[\delta,d]$. For the base case $\ell=\delta$, just note that $B_{\delta-1}=\emptyset$. Now, let $\ell>\delta$, then $\Shad_{\bf t}(B_{\ell-2})\subseteq L_{\ell-1}$ by the inductive hypothesis. Hence,
	\begin{align*}
	\Shad_{\bf t}(B_{\ell-1})\ &=\ \Shad_{\bf t}(L_{\ell-1}\cup\Shad_{\bf 0}(B_{\ell-2}))\\
	&=\ \Shad_{\bf t}(L_{\ell-1}\cup\Shad_{\bf t}(B_{\ell-2}))\\
	&=\ \Shad_{\bf t}(L_{\ell-1}).
	\end{align*}
	Indeed, 
	$$
	\Shad_{\bf t}(\Shad_{\bf 0}(B_{\ell-2}))=\Shad_{\bf t}(\Shad_{\bf t}(B_{\ell-2})).
	$$
	It is clear that the second set is included in the first one. For the other inclusion, let $u\in\Shad_{\bf t}(\Shad_{\bf 0}(B_{\ell-2}))$. Then, by Lemma \ref{Lemma:ShadVectSS} $u=vx_ix_j$ with $i\in[\max(v),j]$, 
	$v\in B_{\ell-2}$ and $\deg(v)=\ell-2$. Since $u$ is ${\bf t}$-spread and $\max(u)=j$, then $vx_i$ is {\bf t}-spread as well. Hence, $vx_i\in\Shad_{\bf t}(B_{\ell-2})$ and so $u\in\Shad_{\bf t}(\Shad_{\bf t}(B_{\ell-2}))$.

	Thus, it remains to prove that $\Shad_{\bf t}(L_{\ell-1})\subseteq L_{\ell}$. Both sets are {\bf t}-spread lex sets. Therefore, the previous inclusion holds if and only if $|\Shad_{\bf t}(L_{\ell-1})|\le|L_{\ell}|$. By Lemma \ref{Lemma:m_i}(b) and Theorem \ref{Thm:BayerVectSpread} applied to the sets $L_{\ell-1}$ and $[I_{\ell-1}]_{\bf t}$ satisfying $|L_{\ell-1}|=|[I_{\ell-1}]_{\bf t}|$, we have,
	\begin{align*}
	\left\lvert \Shad_{\bf t}(L_{\ell-1}) \right\rvert\ &= \sum_{i=1+\sum_{j=1}^{\ell-2} t_j}^{n-t_{\ell-1}} m_{\leq i}(L_{\ell-1})\le\sum_{i=1+\sum_{j=1}^{\ell-2} t_j}^{n-t_{\ell-1}} m_{\leq i}([I_{\ell-1}]_{\bf t})\\
	&=\ |\Shad_{\bf t}([I_{\ell-1}]_{\bf t})|\le|[I_{\ell}]_{\bf t}|=|L_\ell|.
	\end{align*}
	The inductive proof is complete. 
	
	We denote $J$ by $I^{{\bf t},\lex}$. It is clear that $I^{{\bf t},\lex}$ is the unique ideal meeting the requirements of the statement.
\end{proof}

\begin{Example}\em\label{Ex:Itlex}
	Let ${\bf t}=(1,0,2)$ and $n=6$. Consider the following ${\bf t}$-spread strongly stable ideal of $S=K[x_1,\dots,x_6]$:
	$$
	I=(x_1x_2,x_1x_3,x_1x_4,x_2x_3,x_2x_4^2,x_3x_4^2x_6).
	$$
	Then,
	\begin{align*}
	[I_\ell]_{\bf t}\ &=\ \emptyset,\ \text{for}\ \ell=0,1,\\[0.3em]
	[I_2]_{\bf t}\ &=\ \{x_1x_2,x_1x_3,x_1x_4,x_2x_3\},\\[0.3em]
	[I_3]_{\bf t}\ &=\ \{x_1x_2^2,x_1x_2x_3,x_1x_2x_4,x_1x_2x_5,x_1x_2x_6,x_1x_3^2,x_1x_3x_4,x_1x_3x_5,x_1x_3x_6,\\
	&\phantom{=\ \{.}x_1x_4^2,x_1x_4x_5,x_1x_4x_6,x_2x_3^2,x_2x_3x_4,x_2x_3x_5,x_2x_3x_6,x_2x_4^2\},\\[0.3em]
	[I_4]_{\bf t}\ &=\ \{x_1x_2^2x_4,x_1x_2^2x_5,x_1x_2^2x_6,x_1x_2x_3x_5,x_1x_2x_3x_6,x_1x_2x_4x_6,x_1x_3^2x_5,x_1x_3^2x_6,\\
	&\phantom{=\ \{.}x_1x_3x_4x_6,x_1x_4^2x_6,x_2x_3^2x_5,x_2x_3^2x_6,x_2x_3x_4x_6,x_2x_4^2x_6,x_3x_4^2x_6\},\\[0.3em]
	[I_\ell]_{\bf t}\ &=\ \emptyset,\ \text{for all}\ \ell\ge5.
	\end{align*}
	Therefore,
	\begin{align*}
	f_{\bf t}(I)\ &=\ (f_{{\bf t},-1}(I),f_{{\bf t},0}(I),f_{{\bf t},1}(I),f_{{\bf t},2}(I),f_{{\bf t},3}(I))\\
	&=\ (1,6,11,18,0).
	\end{align*}
		
	Note that the value of $f_{{\bf t},3}(I)$ depends on the fact that $[I_4]_{\bf t}=M_{6,4,{\bf t}}$.
	
	Moreover, $L_\ell=\emptyset$ for $\ell=0,1$ and for $\ell\ge5$.
	Whereas, for $\ell=2,3,4$, we have
	\begin{align*}
	L_2\ &=\ \{x_1x_2,x_1x_3,x_1x_4,x_1x_5\},\\[0.3em]
	L_3\ &=\ \{x_1x_2^2,x_1x_2x_3,x_1x_2x_4,x_1x_2x_5,x_1x_2x_6,x_1x_3^2,x_1x_3x_4,x_1x_3x_5,x_1x_3x_6,\\
	&\phantom{=\ \{.}x_1x_4^2,x_1x_4x_5,x_1x_4x_6,x_1x_5^2,x_1x_5x_6,x_1x_6^2,x_2x_3^2,x_2x_3x_4\},\\[0.3em]
	L_4\ &=\ \{x_1x_2^2x_4,x_1x_2^2x_5,x_1x_2^2x_6,x_1x_2x_3x_5,x_1x_2x_3x_6,x_1x_2x_4x_6,x_1x_3^2x_5,x_1x_3^2x_6,\\
	&\phantom{=\ \{.}x_1x_3x_4x_6,x_1x_4^2x_6,x_2x_3^2x_5,x_2x_3^2x_6,x_2x_3x_4x_6,x_2x_4^2x_6,x_3x_4^2x_6\}.
	\end{align*}
	Hence,
	$$
	I^{{\bf t},\lex}=(x_1x_2,x_1x_3,x_1x_4,x_1x_5,x_1x_6^2,x_2x_3^2,x_2x_3x_4,x_2x_4^2x_6,x_3x_4^2x_6).
	$$
\end{Example}

\section{The vector-spread Macaulay theorem}\label{sec3}
The purpose of this section is to give a classification of all possible $f_{\bf t}$-vectors of a ${\bf t}$-spread strongly stable ideal. We follow the steps of the classical Macaulay theorem, see \cite[Theorem 6.3.8]{JT}.
	
We quote the following result from \cite[Lemma 6.3.4]{JT}.
\begin{Lemma}\label{Lemma:BinExp}
	Let $\ell$ be a positive integer. Then, each positive integer $a$ has a unique expansion
	$$
	a=\binom{a_\ell}{\ell}+\binom{a_{\ell-1}}{\ell-1}+\dots+\binom{a_p}{p},
	$$
	with $a_\ell>a_{\ell-1}>\dots>a_{p}\ge p\ge1$.
\end{Lemma}

The previous expansion is called the {\em binomial expansion} (or {\em Macaulay expansion}) of $a$ with respect to $\ell$. Hereafter, suppose we can write a positive integer $a$ as
\begin{equation}\label{eq:fakeBinExp}
a=\binom{a_{\ell}}{\ell}+\dots+\binom{a_p}{p}+\dots+\binom{a_1}{1},
\end{equation}
where $a_\ell>a_{\ell-1}>\dots>a_p\ge p$ and $a_j<j$ for $j=1,\dots,p-1$. Then, by Lemma \ref{Lemma:BinExp}, $a=\sum_{j=p}^{\ell}\binom{a_j}{j}$ is the (unique) binomial expansion of $a$ with respect to $\ell$. However, for our convenience, we refer to \eqref{eq:fakeBinExp} also as a binomial expansion of $a$.\smallskip
	
\begin{Definition}\em\label{Def:vector-spreadOp}
	Let $n,\ell$ be positive integers, ${\bf t}=(t_1,\dots,t_{d-1})\in\ZZ_{\ge0}^{d-1}$, $d\ge2$ such that $n>\sum_{j=1}^{d-1}t_j$ and $\ell<d$. For all $\ell\in[1,d-1]$, we define a \textit{{\bf t}-spread operator} as follows: for any positive integer $a\le|M_{n,\ell,{\bf t}}|$, let $a=\sum_{j=p}^{\ell}\binom{a_j}{j}$ be the binomial expansion of $a$ with respect to $\ell$. We define
	$$
	a^{(\ell,{\bf t})}=\sum_{j=p+1}^{\ell+1}\binom{a_{j-1}+1-t_\ell}{j}.
	$$
\end{Definition}
	
Let $u\in M_{n,\ell,{\bf t}}$. We define the \textit{initial ${\bf t}$-spread lexsegment set determined by $u$} to be the set
\begin{align*}
\mathcal{L}_{\bf t}^i(u)\ &=\ \{v\in M_{n,\ell,{\bf t}}:v\ge_{\lex}u\}.
\end{align*}
Note that any ${\bf t}$-spread lex set $L\subset M_{n,\ell,{\bf t}}$ is an initial ${\bf t}$-spread lexsegment set.\smallskip
	
Definition \ref{Def:vector-spreadOp} is justified by the next result.

\begin{Theorem}\label{Thm:M{n,l,t}VectorOp}
	Let $u\in M_{n,\ell,{\bf t}}$ with $\ell<d$ and $a=|M_{n,\ell,{\bf t}}\setminus\mathcal{L}_{\bf t}^i(u)|$. Then,
	$$
	\big|M_{n,\ell+1,{\bf t}}\setminus\Shad_{\bf t}(\mathcal{L}_{\bf t}^i(u))\big|\ =\ a^{(\ell,{\bf t})}.
	$$
\end{Theorem}
	
In order to prove the theorem, we need some preliminary lemmata.\smallskip
	
Given $\emptyset\neq A \subseteq [1,n]$, we set
\[
M_{A,\ell,{\bf t}} = M_{n,\ell,{\bf t}} \cap K[x_a : a\in A].
\]
Moreover, if ${\bf t}=(t_1,\dots,t_{d-1})\in\ZZ_{\ge0}^{d-1}$, we set ${\bf t}_{\ge k}=(t_k,\dots,t_{d-1})$.

\begin{Lemma}\label{Lemma:utilda}
	Let $u=x_{i_1}\cdots x_{i_{\ell}}\in M_{n,\ell,{\bf t}}$. Then
	\begin{equation}\label{Lemma:M-Lu}
		M_{n,\ell,{\bf t}} \setminus \mathcal{L}_{\bf t}^i(u) = \Union_{k=1}^{\ell} x_{i_1}\cdots x_{i_{k-1}} M_{[i_{k}+1,n],\ell-(k-1),\, {\bf t}_{\ge k}}.
	\end{equation}
	This union is disjoint, and the binomial expansion of $\left\lvert M_{n,\ell,{\bf t}} \setminus \mathcal{L}_{\bf t}^i(u)\right\rvert$ is
	\begin{equation}\label{Lemma:|M-Lu|}
	     \left\lvert M_{n,\ell,{\bf t}} \setminus \mathcal{L}_{\bf t}^i(u)\right\rvert = \sum_{j=1}^{\ell} \binom{a_j}{j},
	\end{equation}
	where $a_j=n-i_{\ell-(j-1)}+j-1-\sum_{h=\ell-(j-1)}^{\ell-1}t_h$, for all $j\in[1,\ell]$.
\end{Lemma}
\begin{proof}
	Since $\geq_{\lex}$ is a total order, we have $M_{n,\ell,{\bf t}} \setminus \mathcal{L}_{\bf t}^i(u) = \left\{v\in M_{n,\ell,{\bf t}}: v<_{\lex}u\right\}$.\linebreak
	Let $v=x_{j_1}\cdots x_{j_\ell} \in M_{n,\ell,{\bf t}}$, with $v<_{\lex} u$. Then, $i_1=j_1,\ldots,i_{k-1}=j_{k-1}$ and $i_k<j_k$, for some $k\in [1,\ell]$. Hence, $v=x_{i_1}\cdots x_{i_{k-1}} w$, where $w\in M_{[i_{k}+1,n],\ell-(k-1), {\bf t}_{\ge k}}$ and \eqref{Lemma:M-Lu} follows.
	
	To prove \eqref{Lemma:|M-Lu|} one can apply \eqref{Formula:|Mn,l,t|}, observing that the union in \eqref{Lemma:M-Lu} is disjoint and that $\lvert x_{i_1}\cdots x_{i_{k-1}}M_{[i_{k}+1,n],\ell-(k-1), {\bf t}_{\ge k}} \rvert = \lvert M_{n-i_k,\ell-(k-1), {\bf t}_{\ge k}} \rvert$. In fact,
	\begin{align*}
	\left\lvert M_{n,\ell,{\bf t}} \setminus \mathcal{L}_{\bf t}^i(u)\right\rvert & = \Big\lvert\Union_{k=1}^{\ell} x_{i_1}\cdots x_{i_{k-1}} M_{[i_{k}+1,n],\ell-(k-1),\, {\bf t}_{\ge k}}\Big\rvert \\
	 &= \sum_{k=1}^{\ell} \left\lvert M_{n-i_k,\ell-(k-1), {\bf t}_{\ge k}} \right\rvert \\
	 &= \sum_{k=1}^{\ell} \binom{n-i_k +(\ell-(k-1)-1)-\sum_{h=k}^{\ell-1}t_h}{\ell-(k-1)}\\
	 &= \sum_{j=1}^{\ell} \binom{n-i_{\ell-(j-1)}+j-1-\sum_{h=\ell-(j-1)}^{\ell-1}t_h}{j},
	\end{align*}
	where in the last equality we set $j=\ell-(k-1)$.
	
	It remains to prove that \eqref{Lemma:|M-Lu|} is the binomial expansion of $\left\lvert M_{n,\ell,{\bf t}} \setminus \mathcal{L}_{\bf t}^i(u)\right\rvert$. Let $p=\min\{j:a_j\ge j\}$. By Lemma \ref{Lemma:BinExp}, it is enough to show the following facts:
	\begin{enumerate}
		\item[(i)] $a_\ell>a_{\ell-1}>\dots>a_p\ge p$, and
		\item[(ii)] $a_j<j$, for $j=1,\dots,p-1$.
	\end{enumerate}
	
	Statement (ii) follows from the definition of $p$. For the proof of (i), let $\ell>j\ge p$. Then, we have
	$$
	a_{j+1}-a_{j}=i_{\ell-(j-1)}-i_{\ell-j}+1-t_{\ell-j}\ge t_{\ell-j}+1-t_{\ell - j}=1,
	$$
	since $i_{\ell-(j-1)}-i_{\ell-j}\ge t_{\ell-j}$. Thus,
	$$
	a_{j+1}\ge a_j+1,
	$$
	and so $a_\ell>a_{\ell-1}>\dots>a_p\ge p$, as desired.
\end{proof}
	
Let $L\subseteq M_{n,\ell,{\bf t}}$ be a ${\bf t}$-spread lex set, with $\ell<d$. By Proposition \ref{Prop:Shadt(L)}(b), $\Shad_{\bf t}(L)\subseteq M_{n,\ell+1,{\bf t}}$ is again a ${\bf t}$-spread lex set. Let
$$
u=\min_{>_{\lex}}L=x_{i_1}x_{i_2}\cdots x_{i_\ell}.
$$
Then, $L=\mathcal{L}_{\bf t}^i(u)$. Hence, if we set $\widetilde{L}=\Shad_{\bf t}(L)$ and $\widetilde{u}=\min_{>_{\lex}}\Shad_{\bf t}(L)$, then $\widetilde{L}= \mathcal{L}_{\bf t}^i(\widetilde{u})$. Therefore, to determine the ${\bf t}$-spread shadow $\widetilde{L}$ of $L$ it is enough to determine the monomial $\widetilde{u}$. This is accomplished in the next lemma.

\begin{Lemma}\label{Lemma:utildavect}
	With the notation and assumptions as above, we have
	\begin{equation}\label{eq:utilda}
	\widetilde{u}=\big(\prod_{m=1}^{\ell-r}x_{i_m}\big)\big(\prod_{j=1}^{r+1}x_{n-\sum_{p=\ell-(r-j)}^{\ell}t_p}\big),
	\end{equation}
	where we set $i_0=t_0=0$ and
	\begin{equation}\label{eq:rwidetilde(u)}
	r=\min\Big\{s\in [0,\ell] : n-i_\ell+\sum_{h=1}^{s}(i_{\ell-(h-1)}-i_{\ell-h}-t_{\ell-h})\ge t_\ell \Big\}.
	\end{equation}
\end{Lemma} 
\begin{proof}
	Let us prove that $\widetilde{u}$ belongs to $\widetilde{L}$. For this aim, it is enough to show that $v=\widetilde{u}/x_n\in L$. Note that 
	$$
	n-i_\ell+\sum_{h=1}^{s}(i_{\ell-(h-1)}-i_{\ell-h}-t_{\ell-h})=n-i_{\ell-s}-\sum_{h=1}^s t_{\ell-h},
	$$
	for all $s\in[0,\ell]$. Thus, $r=\min\big\{s\in [0,\ell] : n-i_{\ell-s}-\sum_{h=1}^s t_{\ell-h}\ge t_\ell \big\}$. Hence,
	$$
	\big(n-\sum_{h=1}^r t_{\ell-h}\big)-i_{\ell-r}\ge t_\ell.
	$$
	By definition of $>_{\lex}$ we have $v\ge_{\lex}u$. Since $v$ is ${\bf t}$-spread, it follows that $v\in L$.
	
	To prove that $\widetilde{u}=\min_{>_{\lex}}\widetilde{L}$, suppose by contradiction that there exists $w\in\widetilde{L}$ such that $w<_{\lex}\widetilde{u}$. Write $w=x_{j_1}\cdots x_{j_{\ell+1}}$, $\widetilde{u}=x_{k_1}\cdots x_{k_{\ell+1}}$. Then, $j_1=k_1,\dots,j_{q-1}=k_{q-1}$ and $j_q>k_q$, for some $q\in[1,\ell+1]$.
	
	If $q\ge \ell-r+1$, then $j_q>k_q=n-\sum_{p=q}^{\ell}t_p$. This is absurd, because all monomials $x_{s_1}\cdots x_{s_{\ell+1}}\in M_{n,\ell+1,{\bf t}}$ satisfy the inequalities $s_{q}\le n-\sum_{h=q}^{\ell}t_h$, $q\in[1,\ell+1]$.
	
	If $1\le q\le \ell-r$, then $j_q>k_q=i_q$. By Lemma \ref{Lemma:ShadVectSS}, $w/x_{j_{\ell+1}}=w'\in L$. Hence, $\min_{>_{\lex}} L=u>_{\lex}w'$, a contradiction. Finally, $\widetilde{u}=\min_{>_{\lex}}\widetilde{L}$.
\end{proof}

The next example illustrates the previous lemma.
\begin{Example}\em
	Let ${\bf t}=(2,1,2)$, $S=K[x_1,\dots,x_8]$, $L=\mathcal{L}^i(u)$ for some $u\in M_{8,3,{\bf t}}$. Set $\Shad_{\bf t}(L)=\widetilde{L}$ and $\widetilde{u}=\min_{>_{\lex}}\widetilde{L}$. Let $r$ the integer defined in (\ref{eq:rwidetilde(u)}).\smallskip

	Let $u=x_2x_4x_6$. Since $n-\max(u)=2=t_3$, then $r=0$ and $\widetilde{u}=ux_n=x_2x_4x_6x_8$.\smallskip

	Let $u=x_2x_6x_7$. Then, $r=2$ and $\widetilde{u}=x_2x_5x_6x_8$.\smallskip

	Let $u=x_4x_6x_7$. Then, $r=3$ and $\widetilde{u}=x_3x_5x_6x_8$. In such a case $\Shad_{\bf t}(L)=M_{8,4,{\bf t}}$.
\end{Example}

\begin{proof}[Proof of Theorem \ref{Thm:M{n,l,t}VectorOp}]
	As before, let $L=\mathcal{L}_{\bf t}^i(u)$, $\widetilde{L}=\Shad_{\bf t}(L)$ and $\widetilde{u}=\min_{>_{\lex}}\widetilde{L}$. Write $\widetilde{u}=x_{k_1}x_{k_2}\cdots x_{k_{\ell+1}}$, where the indices $k_j$ are determined in (\ref{eq:utilda}) and
	$$
	r=\min\Big\{s\in [0,\ell] : n-i_\ell+\sum_{h=1}^{s}(i_{\ell-(h-1)}-i_{\ell-h}-t_{\ell-h})\ge t_\ell \Big\}.
	$$
	
	Then, by Lemma \ref{Lemma:utilda}, we have the binomial expansions
	\begin{align*}
	|M_{n,\ell,{\bf t}}\setminus L|\ =\ \sum_{j=1}^{\ell}\binom{a_j}{j},\ \ \ \ \ \ \ \	
	|M_{n,\ell+1,{\bf t}}\setminus\widetilde{L}|\ =\ \sum_{j=1}^{\ell+1}\binom{\widetilde{a}_j}{j},
	\end{align*}
	where
	\begin{enumerate}
	\item[(a)] $a_j=n-i_{\ell-(j-1)}+j-1-\sum_{h=\ell-(j-1)}^{\ell-1}t_h$, for all $j\in[1,\ell]$, and
	\item[(b)] $\widetilde{a}_j=n-k_{\ell+1-(j-1)}+j-1-\sum_{h=\ell+1-(j-1)}^{\ell}t_h$, for all $j\in[1,\ell+1]$.
	\end{enumerate}
	
	It remains to prove that $|M_{n,\ell+1,{\bf t}}\setminus\widetilde{L}|=a^{(\ell,{\bf t})}$. Firstly, we establish how the coefficients $a_j$ and $\widetilde{a}_j$ are related. Note that, for $j\in[1,r+1]$, we have
	\begin{align*}
	\widetilde{a}_j&=n-k_{\ell+1-(j-1)}+j-1-\sum_{h=\ell+1-(j-1)}^{\ell}t_h\\
	&=n-\Big(n-\sum_{p=\ell+1-(j-1)}^{\ell}t_p\Big)+j-1-\sum_{h=\ell+1-(j-1)}^{\ell}t_h\\
	&=j-1.
	\end{align*}
	Since $\binom{j-1}{j}=0$, we may write as well
	$$
	|M_{n,\ell+1,{\bf t}}\setminus\widetilde{L}|=\sum_{j=r+2}^{\ell+1}\binom{\widetilde{a}_j}{j}.
	$$
	Instead, since $k_{\ell+1-(j-1)}=i_{\ell-(j-1)}$, for $j\in[r+2,\ell+1]$, we have
	$$
	\widetilde{a}_j=a_{j-1}+1-t_\ell.
	$$
	Therefore,
	\begin{align*}
	\big| M_{n,\ell+1,{\bf t}} \setminus\widetilde{L}\big| &= \sum_{j=r+2}^{\ell+1} \binom{a_{j-1}+1-t_{\ell}}{j}.
	\end{align*}
	Let $p=\min\{j:a_j\ge j\}$. The theorem is proved if we show that
	$$
	\big| M_{n,\ell+1,{\bf t}} \setminus\widetilde{L}\big|=\big| M_{n,\ell,{\bf t}} \setminus L\big|^{(\ell,{\bf t})}=\sum_{j=p+1}^{\ell+1}\binom{a_{j-1}+1-t_{\ell}}{j}.
	$$
	If $p+1=r+2$ this is clear. Suppose $p+1>r+2$. Then, it is enough to show that
	$$
	\binom{a_{j-1}+1-t_{\ell}}{j}=0, \ \ \ \text{for all}\ j\in[r+2,p].
	$$
	If $j\le p$, then $a_{j-1}<j-1$. Hence, $a_{j-1}+1-t_\ell\le a_{j-1}+1<j$ and $\binom{a_{j-1}+1-t_{\ell}}{j}=0$, as desired. Now let 
	$r+2>p+1$. We must prove that
	$$
	\binom{a_{j-1}+1-t_{\ell}}{j}=0, 
	$$
	for all $j\in [p+1,r+1]$. 
	Set $a_{\ell+1}=n-\sum_{j=1}^{\ell-1}t_j$. Then
	$$
	r=\min\{s\in[0,\ell]:a_{s+1}-s\ge t_\ell\}.
	$$
	If $j\le r+1$, then $j-2\le r-1$. Hence, $a_{(j-2)+1}-(j-2)=a_{j-1}-(j-2)<t_\ell$. It follows that $a_{j-1}+1-t_\ell<a_{j-1}+2-t_\ell<j$ and $\binom{a_{j-1}+1-t_{\ell}}{j}=0$, as desired.
\end{proof}

\begin{Example}\em
	Let $n=31$, ${\bf t}=(0,1,3,1)$, $a=2023$ and $\ell=3$. Then
	$$
	a=\sum_{j=1}^{\ell}\binom{a_j}{j}=\binom{23}{3}+\binom{22}{2}+\binom{21}{1}
	$$
	is the binomial expansion of $a$ with respect to $\ell$. Therefore, since $r=0$, we have
	\begin{align*}
	a^{(\ell,{\bf t})}=2023^{(3,(0,1,3,1))}\ &=\ \sum_{j=r+2}^{\ell+1}\binom{a_{j-1}+1-t_\ell}{j}\\
	&=\ \binom{19}{2}+\binom{20}{3}+\binom{21}{4}=7296.
	\end{align*}
\end{Example}

Now, we can state and prove the main result in the article.
\begin{Theorem}\label{thm:main}
	Let $f=(f_{-1},f_0,\dots,f_{d-1})$ be a sequence of non-negative integers. The following conditions are equivalent:
	\begin{enumerate}
		\item[\textup{(i)}] there exists a ${\bf t}$-spread strongly stable ideal $I\subset S=K[x_1,\dots,x_n]$ such that $$f_{\bf t}(I)=f;$$
		\item[\textup{(ii)}] $f_{-1}=1$ and $f_{\ell+1}\le f_{\ell}^{(\ell+1,{\bf t})}$, for all $\ell=-1,\dots,d-2$.
	\end{enumerate}
\end{Theorem}
\begin{proof}
	(i) $\implies$ (ii). Assume that $f_{\bf t}(I)=f$. By Proposition \ref{Cor:SubstituteTLex}, we may replace $I$ by $I^{{\bf t},\lex}$ without changing the $f_{\bf t}$-vector. Thus, we may assume as well that $I$ is a ${\bf t}$-spread lex ideal. Then, $f_{-1}=f_{{\bf t},-1}(I)=1$ and for all $\ell\in[-1,d-2]$ we have $\Shad_{\bf t}([I_{\ell+1}]_{\bf t})\subseteq[I_{\ell+2}]_{\bf t}$. Hence,
	\begin{align*}
	f_{\ell+1}=f_{{\bf t},\ell+1}(I)=|M_{n,\ell+2,{\bf t}}|-|[I_{\ell+2}]_{\bf t}|\ &\le\ |M_{n,\ell+2,{\bf t}}|-|\Shad_{\bf t}([I_{\ell+1}]_{\bf t})|\\&=\ |M_{n,\ell+2,{\bf t}}\setminus\Shad_{\bf t}([I_{\ell+1}]_{\bf t})|\\&=\ f_{{\bf t},\ell}(I)^{(\ell+1,{\bf t})}=f_\ell^{(\ell+1,{\bf t})},
	\end{align*}
	where the last equality follows from Theorem \ref{Thm:M{n,l,t}VectorOp}. Statement (ii) is proved.\smallskip
	
	\noindent
	(ii) $\implies$ (i). First we prove that
	$$
	f_{\ell+1}\ \le\ |M_{n,\ell+2,{\bf t}}|,\ \ \ \text{for all}\ \ \ \ell=-2,\dots,d-2.
	$$
	For $\ell=-2$, $f_{-1}=1=|M_{n,0,{\bf t}}|$ because there is only one ${\bf t}$-spread monomial of degree 0, namely $u=1$. Now we proceed by induction. Let $\ell \geq -1$. By the hypothesis (ii), we have $f_{\ell+1}\le f_{\ell}^{(\ell+1,{\bf t})}$ and, by induction, $f_{\ell}\le|M_{n,\ell+1,{\bf t}}|$. Thus, there exists a unique monomial $u\in M_{n,\ell+1,{\bf t}}$ such that $|M_{n,\ell+1,{\bf t}}\setminus\mathcal{L}_{\bf t}^i(u)|=f_{\ell}$. By Theorem \ref{Thm:M{n,l,t}VectorOp}, we have $f_{\ell}^{(\ell+1,{\bf t})}=|M_{n,\ell+2,{\bf t}}\setminus\Shad_{\bf t}(\mathcal{L}_{\bf t}^i(u))|$. This shows that $f_{\ell}^{(\ell+1,{\bf t})}\le|M_{n,\ell+2,{\bf t}}|$ and consequently we have $f_{\ell+1}\le|M_{n,\ell+2,{\bf t}}|$, as desired.
	
	For all $\ell\in [0,d]$, let $L_\ell$ be the unique ${\bf t}$-spread lex set of $M_{n,\ell,{\bf t}}$ such that\linebreak $|L_{\ell}|=|M_{n,\ell,{\bf t}}|-f_{\ell-1}$. For $\ell>d$ we set $L_{\ell}=\emptyset$. As in Proposition \ref{Cor:SubstituteTLex}, we construct the ideal $I=\bigoplus_{\ell\ge0}I_\ell$ where $I_\ell$ is the $K$-vector space spanned by the set
	$$
	L_\ell\cup\Shad_{\bf 0}(B_{\ell-1}),
	$$
	where $B_{-1}=\emptyset$ and for $\ell \geq 1$, $B_{\ell-1}$ is the set of monomials generating $I_{\ell-1}$. As in Proposition \ref{Cor:SubstituteTLex}, one shows that $I$ is a {\bf t}-spread lex ideal. Hence, it remains to prove that $f_{\bf t}(I)=f$. As in the proof of Proposition \ref{Cor:SubstituteTLex}, this boils down to proving that $\Shad_{\bf t}(L_{\ell+1})\subseteq L_{\ell+2}$, for all $\ell\in [-1,d-2]$. Since $f_{\ell+1}\le f_{\ell}^{(\ell+1,{\bf t})}$ we have
	$$
	|M_{n,\ell+2,{\bf t}}\setminus L_{\ell+2}|\le|M_{n,\ell+1,{\bf t}}\setminus L_{\ell+1}|^{(\ell+1,{\bf t})}=|M_{n,\ell+2,{\bf t}}\setminus\Shad_{\bf t}(L_{\ell+1})|,
	$$
	where the last equality follows from Theorem \ref{Thm:M{n,l,t}VectorOp}. Thus, $|\Shad_{\bf t}(L_{\ell+1})|\le|L_{\ell+2}|$. Hence, $\Shad_{\bf t}(L_{\ell+1})\subseteq L_{\ell+2}$, because both are ${\bf t}$-spread lex sets. The proof is complete.
\end{proof}

\begin{Example}\em
	Let ${\bf t}=(1,0,2)$, $d=4$ and $n=6$. Consider the following vector
	\begin{align*}
	f\ &=\ (f_{-1},f_{0},f_{1},f_{2},f_{3})\ =\ (1,6,11,18,0).
	\end{align*}
	Then, $f_{-1}=1$ and $f_{\ell+1}\le f_{\ell}^{(\ell+1,{\bf t})}$, for all $\ell=-1,\dots,2$. Therefore, from Theorem \ref{thm:main} there exists a ${\bf t}$-spread strongly stable ideal of $S=K[x_1,\dots,x_6]$ that has $f$ as a $f_{\bf t}$-vector. The ideal $I$ of Example \ref{Ex:Itlex} is such an ideal.
\end{Example}

\section{An application}\label{sec4}
In this final section, as an application we recover the vector-spread version of the well--known result proved by Bigatti \cite{BAM} and Hulett \cite{HH}, independently (see, also, \cite{HH2, KP}). More precisely, we prove that in the class of all ${\bf t}$-spread strongly stable ideals with a given $f_{\bf t}$-vector, the ${\bf t}$-spread lex ideals have the largest graded Betti numbers.

\begin{Theorem}\label{thm:upperbound}
	Let  $I\subset S=K[x_1,\dots,x_n]$ be a ${\bf t}$-spread strongly stable ideal. Then,
	$$
	\beta_{i,j}(I)\ \le\ \beta_{i,j}(I^{{\bf t},\lex}), \ \ \ \text{for all}\ i\ \text{and}\ j.
	$$
\end{Theorem}
\begin{proof}
	By \cite[Corollary 5.2]{F1}, 
	we have
	\begin{equation}\label{eneherzogqureshiformulabetti}
	\beta_{i,i+j}(I)\ =\ \sum_{u\in G(I)_j}\binom{\max(u)-1-\sum_{h=1}^{j-1}t_h}{i}.
	\end{equation}
	We are going to write (\ref{eneherzogqureshiformulabetti}) in a more suitable way. We observe that $I$ is a {\bf t}-spread ideal and thus 
	$$
	G(I)_j\ =\ [I_j]_{\bf t}\setminus \Shad_{\bf t}([I_{j-1}]_{\bf t}).
	$$
	Hence, we can write the Betti number in (\ref{eneherzogqureshiformulabetti}) as a difference $A-B$, where
	\begin{align*}
	A\ =&\ \sum_{u\in G([I_j]_{\bf t})} \binom{\max(u)-1-\sum_{h=1}^{j-1}t_h}{i}=\sum_{k=1}^n m_k([I_j]_{\bf t})\binom{k-1-\sum_{h=1}^{j-1}t_h}{i}\\
	=&\ \sum_{k=1}^n\Big(m_{\le k}([I_j]_{\bf t})-m_{\le k-1}([I_j]_{\bf t})\Big)\binom{k-1-\sum_{h=1}^{j-1}t_h}{i}\\
	=&\ \sum_{k=1}^nm_{\le k}([I_j]_{\bf t})\binom{k-1-\sum_{h=1}^{j-1}t_h}{i}-\sum_{k=1}^{n-1}m_{\le k}([I_j]_{\bf t})\binom{k-\sum_{h=1}^{j-1}t_h}{i}
	\end{align*}
	and
	\begin{align*}
	B\ =&\ \sum_{u\in \Shad_{\bf t}([I_{j-1}]_{\bf t})}\binom{\max(u)-1-\sum_{h=1}^{j-1}t_h}{i}\\=&\ \sum_{k=1+\sum_{h=1}^{j-1}t_h}^n\!\!\!\!m_k(\Shad_{\bf t}([I_{j-1}]_{\bf t}))\binom{k-1-\sum_{h=1}^{j-1}t_h}{i}\\=&\ \sum_{k=1+\sum_{h=1}^{j-1}t_h}^n m_{\le k-t_{j-1}}([I_{j-1}]_{\bf t})\binom{k-1-\sum_{h=1}^{j-1}t_h}{i},
	\end{align*}
	where the last equality follows from Lemma \ref{Lemma:m_i}(a).
	
	Furthermore, we can write $A=A_1-A_2$ with
	\begin{align*}
	A_1\ =&\ m_{\le n}([I_j]_{\bf t})\binom{n-1-\sum_{h=1}^{j-1}t_h}{i},\\[0.8em]
	A_2\ =&\ \sum_{k=1}^{n-1}m_{\le k}([I_j]_{\bf t})\bigg[\binom{k-\sum_{h=1}^{j-1}t_h}{i}-\binom{k-1-\sum_{h=1}^{j-1}t_h}{i}\bigg]\\
	=&\ \sum_{k=1+\sum_{h=1}^{j-1}t_h}^{n-1}m_{\le k}([I_j]_{\bf t})\binom{k-1-\sum_{h=1}^{j-1}t_h}{i-1}.
	\end{align*}
	
	Therefore, we obtain
	\begin{equation}\label{presentationbettinumbers}
	\begin{aligned}
	\beta_{i,i+j}(J)\ &=\ m_{\le n}([I_j]_{\bf t})\binom{n-1-\sum_{h=1}^{j-1}t_h}{i}\\
	&-\sum_{k=1+\sum_{h=1}^{j-1}t_h}^{n-1}m_{\le k}([I_j]_{\bf t})\binom{k-1-\sum_{h=1}^{j-1}t_h}{i-1}\\
	&-\sum_{k=1+\sum_{h=1}^{j-1}t_h}^n m_{\le k-t_{j-1}}([I_{j-1}]_{\bf t})\binom{k-1-\sum_{h=1}^{j-1}t_h}{i}.
	\end{aligned}
	\end{equation}
	
	Now, we compute the graded Betti numbers $\beta_{i,i+j}(I^{{\bf t},\lex})$. 
	Recall that $I$ and $I^{{\bf t},\lex}$ share the same $f_{\bf t}$-vector. Therefore, $|[I^{{\bf t},\lex}_j]_{\bf t}|=|[I_{j}]_{\bf t}|$, for all $j$. 
	Applying Theorem \ref{Thm:BayerVectSpread}, we have $m_{\le k}([I^{{\bf t},\lex}_j]_{\bf t})\le m_{\le k}([I_j]_{\bf t})$ for all $k\in[n]$. Moreover,
	$$
	m_{\le n}([I_j]_{\bf t})=|[I_j]_{\bf t}|=|[I^{{\bf t},\lex}_j]_{\bf t}|=m_{\le n}([I^{{\bf t},\lex}_j]_{\bf t}).
	$$
	Therefore, replacing in (\ref{presentationbettinumbers}), for all $k$ and $j$, every occurrence of $m_{\le k}([I_j]_{\bf t})$ with $m_{\le k}([I^{{\bf t},\lex}_j]_{\bf t})$, we get the Betti number $\beta_{i,i+j}(I^{{\bf t},\lex})$. Finally, $\beta_{i,i+j}(I)\le\beta_{i,i+j}(I^{{\bf t},\lex})$, for all $i,j\ge0$.
\end{proof}
	
\begin{Remark}\em
	Note that in the previous result, we allow $K$ to be an arbitrary field.
\end{Remark}
	
\begin{Example}\em
	Consider again the ${\bf t}$-spread strongly stable ideal $I$ of Example \ref{Ex:Itlex}. Then, the Betti tables of $I$ and $I^{{\bf t},\lex}$ are, respectively,
	$$
	\begin{matrix}
	  & 0 & 1 & 2\\
	\hline
	2:& 4 & 4 & 1 \\
	3:& 1 & 2 & 1 \\
	4:& 1 & 2 & 1
	\end{matrix}\qquad\qquad\qquad
	\begin{matrix}
	   & 0 & 1 & 2 & 3 & 4\\ 
	\hline
	2: & 4 & 6 & 4 & 1 & .\\
	3: & 3 & 7 & 7 & 4 & 1\\
	4: & 2 & 4 & 2 & . & .
	\end{matrix}
	$$
	From these tables we see that $\beta_{i,i+j}(I)\le\beta_{i,i+j}(I^{{\bf t},\lex})$ for all $i$ and $j$.
\end{Example}
\medskip

\subsection*{Acknowledgments}
The authors thank the anonymous referee for his/her careful reading and helpful suggestions, that allowed us to improve the quality of the paper. The authors acknowledge support of the GNSAGA of INdAM (Italy).


\end{document}